\documentclass[11pt]{amsart}

\usepackage{tabu}
\usepackage{amssymb}
\usepackage{multicol, multirow}
\usepackage{url}
\usepackage{comment}

\usepackage[OT2,T1]{fontenc}
\DeclareSymbolFont{cyrletters}{OT2}{wncyr}{m}{n}
\DeclareMathSymbol{\Sha}{\mathalpha}{cyrletters}{"58}

\newtheorem{theorem}{Theorem}[section]
\newtheorem{lemma}[theorem]{Lemma}
\newtheorem{proposition}[theorem]{Proposition}
\newtheorem{corollary}[theorem]{Corollary}

\theoremstyle{definition}
\newtheorem{definition}[theorem]{Definition}

\theoremstyle{remark}

\numberwithin{equation}{section}

\newcommand{\Q}{{\mathbb Q}}
\newcommand{\Z}{{\mathbb Z}}

\keywords{Elliptic curve; Congruent number; Rational Points}
\subjclass[2020]{Primary: 11G05; Secondary: 11D45, 11E04, 11F37 }

\begin{document}


\title[A variant of CNP]{A variant of the congruent number problem}


\author{Jerome T. Dimabayao}
\address{Institute of Mathematics, College of Science, 
University of the Philippines Diliman, Quezon City, 
Philippines}
\email{jdimabayao@math.upd.edu.ph}


\author{Soma Purkait}
\address{Department of Mathematics, Tokyo Institute of Technology,
Tokyo, 
Japan}
\email{somapurkait@gmail.com}



\begin{abstract}
A positive integer $n$ is called a $\theta$-congruent number 
if there is a triangle with sides $a,b$ and $c$ for which the 
angle between $a$ and $b$ is equal to $\theta$ and its area is 
$n\sqrt{r^2 - s^2}$, where $0 < \theta < \pi$, $\cos \theta = s/r$ 
and $0 \leq |s| < r$ are relatively prime integers.  
The case $\theta=\pi/2$ refers to the classical congruent numbers.
It is known that the problem of classifying $\theta$-congruent numbers 
is related to the existence of rational points on the elliptic curve 
$y^2 = x(x+(r+s)n)(x-(r-s)n)$. 
In this paper, we deal with a variant of the congruent number problem 
where the cosine of a fixed angle is $\pm \sqrt{2}/2$.
\end{abstract}


\maketitle




\setlength\parindent{24pt}
\maketitle

%



\section{Introduction}

A positive integer $n$ is called a \emph{congruent number} if it is the area of a right triangle with rational sides. The congruent number problem asks to determine all congruent numbers and is one of the oldest yet unsolved problems. It is well known that $n$ is congruent if and only if the elliptic curve $y^2 = x^3 - n^2x$ has a positive rank. Tunnell \cite{Tunnell} in his famous work provided a simple criterion for a number to be congruent in terms of ternary quadratic forms. However, his result relied on the Birch and Swinnerton-Dyer conjecture, which is still unproven. In particular it is conjectured that any positive integer $n \equiv 5,6$ and $7 \pmod{8}$ is congruent; this has been partially resolved by a recent work of Tian \cite{Tian}. 
 
There are different variants of the  congruent number problem in literature. Fujiwara \cite{Fujiwara} and Kan \cite{Kan} studied the $\theta$-congruent number problem for an angle $\theta$ with $0 < \theta < \pi$ and rational cosine, while  Yoshida~\cite{Yoshida} considered the special case when $\theta=\pi/3,\ 2\pi/3$. For a nice survey of this and other variants of the congruent number problem please refer to \cite{TopYui}. 
Lario \cite{Lario} considered the variant 
of the $\theta$-congruent number problem  in the case where $\theta = \pi/6$. 
In this case, $\cos \theta = \sqrt{3}/2$ is no longer rational. 
A positive integer $n$ is called $\pi/6$-congruent if there exists a triangle with 
sides $a$, $b\sqrt{3}$, $c$ and angle $\pi/6$ opposite to the side 
$c$ and area equal to $n \sqrt{3}$, where $a,b,c$ are rational numbers. 
The associated $\pi/6$-congruent number elliptic curve 
is given by $C_n: y^2 = x^3 + 6nx^2 - 3n^2 x$. 
It is isogenous to the Mordell curve 
$D_n : y^2 = x^3 + n^3$, which has 
been studied extensively by several mathematicians.

In this work, we study the next simple case where 
$\cos \theta$ is not rational; namely, the case 
$\theta = \pi/4$ or $3 \pi/4$.
We call a positive integer $n$  
\emph{$\pi/4$-congruent} if there exists a 
triangle with sides $a$, $b \sqrt{2}$, 
$c$ and angle $\pi/4$ opposite to the 
side $c$ with $a,b,c$ positive rationals and area $n$. 
It turns out that $n$ is $\pi/4$-congruent if and only if 
the elliptic curve $E_n : y^2 = x^3 + 2nx^2 - n^2x$ 
has infinitely many rational points.
This is derived in Section \ref{CN-EC}. 
In Section \ref{Gen_Birch_Lemma}, we 
list some families of $\pi/4$-congruent numbers with 
arbitrarily many prime factors using a recent 
work of Shu and Zhai \cite{SZ}. 
We compute cusp forms of half-integer weight 
that correspond, via Shimura correspondence, with $E_1$, $E_{-1}$, $E_2$ and $E_{-2}$. 
We apply Waldspurger's theorem to obtain Tunnell-like results, 
allowing us to obtain non-examples of $\pi/4$-congruent primes 
or double primes. These are presented in Section \ref{Waldspurger}. 
In Section \ref{distribution}, we use a simple 
parametrization of $\pi/4$-congruent numbers to show 
that every congruence class modulo a positive integer 
contains infinitely many $\pi/4$-congruent  
numbers.
Finally, in Section \ref{tiling}, we discuss the relation of 
$\pi/4$-congruent numbers with tilings of the square.
Analogous results are given for $\theta=3\pi/4$.

\section{Elliptic curves associated with $\pi/4$- and $3\pi/4$-congruent numbers}\label{CN-EC}

We first recall our definition given in the introduction. 

\begin{definition}
A positive integer $n$ is called 
$\pi/4$-congruent if there exists a 
triangle with sides $a$, $b \sqrt{2}$, 
$c$ and angle $\pi/4$ opposite to the 
side $c$, where $a$,$b$, and $c$ are positive rationals and 
the area is $n$. 
\end{definition}

For example, $2$ is $\pi/4$-congruent as we have the triangle with sides of length 
$1$, $4\sqrt{2}$, $5$, with angle $\pi/4$ 
opposite to the side of length $5$ and area $2$. 
Another example of a $\pi/4$-congruent number is $5$, since the triangle with sides of length 
$\frac{7}{2}$, $\frac{20 \sqrt{2}}{7}$, $\frac{41}{14}$, with angle $\pi/4$ 
opposite to the side of length $41/14$ has area $5$.

By cosine formula we see that 
$n$ is $\pi/4$-congruent if and only if 
there exist positive rationals  $a$, $b$ and $c$ such that 
\begin{equation}
c^2 = a^2 + 2b^2 - 2ab \quad \mbox{ and } 
\quad 2n = ab.
\end{equation}
The above system can be written as 
\begin{equation}
\left(\frac{c}{ab}\right)^2 = \left( \frac{1}{b} \right)^2 + 2\left( \frac{1}{a} \right)^2 - 2
\left( \frac{1}{ab} \right) \quad \mbox{ and } 
\quad 2n\left( \frac{1}{ab} \right) = 1.
\end{equation}
Thus, $n$ is $\pi/4$-congruent if and only if 
there exist positive rationals $a$, $b$ and $c$ such that 
\begin{equation}
c^2 = a^2 + 2b^2 - 2ab \quad \mbox{ and } 
\quad 2nab = 1.
\end{equation}


Note that $n$ is $\pi/4$-congruent if and only if $nk^2$ is $\pi/4$-congruent. 
We assume henceforth that $n$ is square-free. 

Let $E_1$ be the elliptic curve of conductor $128$ given by the Weierstrass form: $y^2=x^3+2x^2-x$. We have the following equivalent notion of $\pi/4$-congruence in terms of quadratic twists of $E_1$. 

\begin{lemma}\label{ellcurve_pi}
A square-free positive integer $n$ is 
$\pi/4$-congruent if and only if the 
elliptic curve $E_n : y^2 = x^3 + 2nx^2 - n^2x$ has a rational point $(x,y)$ with $y\neq 0$. 
\end{lemma}

\begin{proof}
Suppose $n$ is $\pi/4$-congruent. Let 
$a$, $b$ and $c$ be positive rationals such 
that 
$c^2 = a^2 + 2b^2 - 2ab$ and $2n = ab$.
Substitute $b = 2n/a$ in the first equation and multiply by $a^2$ to obtain 
$(ac)^2 = (a^2 -2n)^2 + 4n^2$. 
In particular, $ca + a^2 -2n \neq 0$.
Put $x = (ac + a^2 -2n)/2$ and 
$y=a(ac + a^2 - 2n)/2$. Then 
$(x,y)$ is a rational point on $E_n$, with $y \neq 0$.

Conversely, suppose $(x,y)$ is a rational 
point on $E_n$ with $y \neq 0$. Changing 
the sign of $y$, if necessary, we observe that 
$a = (x^2 + 2nx - n^2)/y > 0$, 
$b = 2nx/y > 0$, and 
$c = (x^2 + n^2)/|y|$ satisfy the 
required relation.
\end{proof}

We can also define the notion of a 
$3 \pi /4$-congruent number. 
A positive integer $n$ is said to be 
$3 \pi /4$-congruent if there exists a 
triangle with sides $a$, $b \sqrt{2}$, 
$c$ and angle $3 \pi/4$ opposite to the 
side $c$, where $a$, $b$, and $c$ are positive rationals and the area is $n$; that is, 
\begin{equation}
c^2 = a^2 + 2b^2 + 2ab \quad \mbox{ and } 
\quad 2n = ab.
\end{equation} 
Equivalently, $n$ is $3 \pi /4$-congruent 
if there exist positive rational numbers $a,b,c$ such that 
\begin{equation}
c^2 = a^2 + 2b^2 + 2ab \quad \mbox{ and } 
\quad 2nab = 1.
\end{equation} 

We have the following analogue of the 
last lemma for the $3 \pi / 4$ case.

\begin{lemma}\label{ellcurve_3pi}
A square-free positive integer $n$ is 
$3\pi/4$-congruent if and only if the 
elliptic curve $E_{-n} : y^2 = x^3 - 2n x^2 - n^2 x$ has a rational point $(x,y)$ with $y\neq 0$. 
\end{lemma}

The elliptic curve $E_1$ does not have complex multiplication and has torsion subgroup $\Z/2\Z$ over $\Q$. We note the following useful lemma for torsion subgroups of $E_n$.

\begin{lemma}\label{lem:tor}
For all non-zero integers $n$, the Mordell-Weil group 
$E_{n}(\mathbb{Q})$ of the quadratic twist $E_n$ has torsion subgroup $\mathbb{Z}/{2 \mathbb{Z}}$.
\end{lemma}

\begin{proof}
The rational $2$-torsion subgroup of $E_n$ is $\mathbb{Z}/{2 \mathbb{Z}}$, so 
Mazur's theorem implies that the 
torsion subgroup of 
$E_{n}(\mathbb{Q})$ is $\mathbb{Z}/{2m \mathbb{Z}}$, where 
$m$ is either $1,2,3,5$ or $6$. 

The $3$-division polynomial of $E_n$ is given by $3x^4 + 8x^3 - 6x^2 - 1$, which 
is irreducible over $\mathbb{Q}$.
Thus, $E_{n}(\mathbb{Q})$ has no element 
of order $3$. Consequently, 
the torsion subgroup is  
neither $\mathbb{Z}/{6 \mathbb{Z}}$ 
nor $\mathbb{Z}/{12 \mathbb{Z}}$.

The $4$-division polynomial of $E_n$ is given by $x(x^2+1)(x^2+2x-1)(x^4 + 4x^3 - 6x^2 -4x +1)$, 
each factor being irreducible in  
$\mathbb{Q}[x]$. This only leads to the 
point $(0,0)$ of order $2$ and we have no 
rational point of order $4$. 

Finally, the $5$-division polynomial of $E_n$ is given by
$5x^{12} + 40x^{11} + 2x^{10} - 160x^9 - 105x^8 - 720x^7 - 660x^6 + 224x^5 + 515x^4 - 280x^3 + 50x^2 + 1$, which is 
irreducible over $\mathbb{Q}$. Thus, 
$E_{n}(\mathbb{Q})$ has no element 
of order $5$. This completes the proof.
\end{proof}

The last three lemmas imply the following result.

\begin{theorem}\label{congruence_MWrank}
A square-free integer $n>1$ is $\pi/4$-congruent 
(resp.\ $3 \pi/4$-congruent)
if and only if the Mordell-Weil group 
$E_n(\mathbb{Q})$ 
(resp.\ $E_{-n}(\mathbb{Q})$) 
has positive rank.  
\end{theorem}

\section{$\pi/4$- and $3\pi/4$-congruent numbers with arbitrarily many prime factors}\label{Gen_Birch_Lemma}

In this section, we present the consequence of the main result 
of Shu and Zhai \cite{SZ} to 
obtain examples of $\pi/4$- and 
$3\pi/4$-congruent numbers with 
arbitrarily many prime factors.  

If $C$ is an elliptic curve over 
$\mathbb{Q}$ of conductor $N$, we let 
$f : X_0(N) \rightarrow C$ be an optimal modular 
parametrization sending the cusp 
$[\infty]$ at infinity to the zero 
element of $C$. Note that $f([0])$ is a 
torsion point in $C(\mathbb{Q})$. 

We define $C'/\mathbb{Q}$ to be the 
quotient curve $C' = C/C[2](\mathbb{Q})$. 
Let (Tor) be the following condition:
\[ C[2](\mathbb{Q}) = C'[2](\mathbb{Q}) = \mathbb{Z}/{2 \mathbb{Z}}. \]
Thus the fields $\mathbb{Q}(C[2])$ and 
$\mathbb{Q}(C'[2])$ are quadratic extensions of $\mathbb{Q}$.

A prime $q$ is said to be 
\emph{admissible} for $C$ if $(q, 2N)=1$ 
and $q$ is inert in both the quadratic 
fields $\mathbb{Q}(C[2])$ and 
$\mathbb{Q}(C'[2])$.

For a prime $p > 3$ such that 
$p \equiv 3 \pmod{4}$, we say that $p$ satisfies the 
\emph{Heegner hypothesis} for $(C,\mathbb{Q}(\sqrt{-p}))$ 
if every prime $\ell$ dividing $N$ 
splits in $\mathbb{Q}(\sqrt{-p})$.
 
For any odd prime $q$, put 
$q^* = \left( \frac{-1}{q} \right) q$. 
 
\begin{theorem}[\cite{SZ}, Theorem 1.2]\label{thm:Birchlem}
Let $C$ be an elliptic curve over 
$\mathbb{Q}$ satisfying $f([0]) \not\in 2C(\mathbb{Q})$ 
and condition (Tor). 
Let $p > 3$ be a prime 
$p \equiv 3 \pmod{4}$ satisfying the 
Heegner hypothesis for $(C,\mathbb{Q}(\sqrt{-p}))$.  
For any integer $r \geq 0$, let 
$q_1, \ldots, q_r$ be distinct admissible 
primes which are not equal to $p$, and 
put $M = q_1^* \cdots q_r^*$. Then we have 
\[ \mathrm{ord}_{s=1} L(C_M, s) = 
\mathrm{rank}_{\mathbb{Q}} (C_M) = 0 \]
and 
\[ \mathrm{ord}_{s=1} L(C_{-pM}, s) = 
\mathrm{rank}_{\mathbb{Q}} (C_{-pM}) = 1, \]
where $C_m$ denotes the $m$-th quadratic twist of $C$.
Moreover, the Tate-Shafarevich groups 
$\Sha(C_M/\mathbb{Q})$ and 
$\Sha(C_{-pM}/\mathbb{Q})$ are both finite.
\end{theorem}

We apply the above theorem to the elliptic curve 
$D : y^2 = x^3 - 4x^2 + 8x$ of conductor $128$ with Cremona label 128D1. 
Note that $D(\Q)$ has torsion subgroup $\Z/2\Z$,  hence for 
the modular parametrization $f : X_0(128) \rightarrow D$, $f([0])=(0,0)$ is a point of order $2$ and so $f([0])\not\in 2D(\mathbb{Q})$. Also, $D'=D/D[2](\mathbb{Q})$ is 
the elliptic curve $E_1 : y^2=x^3+2x^2-x$ with torsion subgroup $\Z/2\Z$. 
Thus, $D$ satisfies the hypothesis of the theorem. 


We have $\mathbb{Q}(D[2]) = \mathbb{Q}(\sqrt{-1})$ and 
$\mathbb{Q}(E_1[2]) = \mathbb{Q}(\sqrt{2})$. 
An odd prime $q$ is inert in both $\mathbb{Q}(\sqrt{2})$ 
and $\mathbb{Q}(\sqrt{-1})$ if and only if $q \equiv 3 \pmod{8}$. 
That is, any prime $q \equiv 3 \pmod{8}$ is admissible for $D$.
Further, any prime $p>3$ such that $p \equiv 7 \pmod{8}$ satisfies the Heegner hypothesis for 
$(D, \mathbb{Q}(\sqrt{-p}))$. Theorem~\ref{thm:Birchlem} and the fact that $D$ and $E_1$ are $2$-isogenous give the following result. 
\begin{corollary}
Let $p$ be a prime with $p \equiv 7 \pmod{8}$. 
For any integer $r \geq 0$, let 
$q_1, \ldots, q_r$ be distinct primes which are not equal to $p$
with $q_j \equiv 3 \pmod{8}$. 
Put $M = (-1)^r q_1 \cdots q_r$. 
Then we have 
\[ \mathrm{ord}_{s=1} L(E_M, s) = 
\mathrm{rank}_{\mathbb{Q}} (E_M) = 0 \]
and 
\[ \mathrm{ord}_{s=1} L(E_{-pM}, s) = 
\mathrm{rank}_{\mathbb{Q}} (E_{-pM}) = 1. \]
Consequently, 
\begin{enumerate}
\item If $r$ is even, then $q_1 \cdots q_r$ is not $\pi/4$-congruent.  
\item If $r$ is odd, then $q_1 \cdots q_r$ is not $3\pi/4$-congruent.  
\item If $r$ is even, then $pq_1 \cdots q_r$ is $3\pi/4$-congruent. 
\item If $r$ is odd, then $pq_1 \cdots q_r$ is $\pi/4$-congruent.
\end{enumerate}
\end{corollary}

\vskip 2mm

We note one more observation. 
In \cite{ST}, Stroeker and Top studied a family 
of elliptic curves of the form $G_p: y^2 = (x+p)(x^2+p^2)$ which are 
quadratic twists of the curve $G: y^2=(x+1)(x^2+1)$. 
It turns out that our curve $E_2$ is $2$-isogenous to $G$. 
Using $2$-descent, the authors proved (\cite[Theorem 1.3]{ST}) 
that the rank of $G_p(\mathbb{Q})$ is zero for a prime 
$p \equiv \pm 3 \pmod{8}$. 
Consequently, we have the following
\begin{corollary}
For a prime $p \equiv \pm 3 \pmod{8}$, 
$2p$ is not $\pi/4$-congruent.
\end{corollary}


\section{Application of Waldspurger's Theorem}\label{Waldspurger}

\subsection{Root numbers}
The root number $W(E/\mathbb{Q})$ gives information on the parity of the 
analytic rank of $E/\mathbb{Q}$. 
We record the following result about  
root numbers $W(E_{\pm n}/\mathbb{Q})$ 
for later use. 
We use Rizzo's tables \cite{Rizzo} to compute these.
\begin{proposition}\label{rootnumbers}
Let $E$ be the elliptic curve 
$y^2 = x^3 + 2x^2 - x$. 
For a positive square-free odd integer $n$ the following hold:
\[ 
W(E_{n}/\mathbb{Q}) =
W(E_{-n}/\mathbb{Q}) =
\begin{cases}
1, &\mbox{if } n \equiv 1, 3 \pmod{8} 
\\
-1, &\mbox{if } n \equiv 5, 7 \pmod{8}.
\end{cases}
\]

\[ 
W(E_{2n}/\mathbb{Q}) =
-W(E_{-2n}/\mathbb{Q}) =
\begin{cases}
1, &\mbox{if } n \equiv 3, 5 \pmod{8} \\
-1, &\mbox{if } n \equiv 1, 7 \pmod{8}.
\end{cases}
\]

\end{proposition}

We have the following immediate consequence.

\begin{corollary}\label{BSD-root}
Let $n > 1$ be a square-free odd integer. Suppose the Birch and Swinnerton-Dyer conjecture holds. 
\begin{enumerate}
\item If $n \equiv 5, 7 \pmod{8}$ then $n$ is a $\pi/4$- 
and a $3 \pi/4$-congruent number. 
\item If $n \equiv 1, 7 \pmod{8}$ then $2n$ is a $\pi/4$-congruent number.
\item If $n \equiv 3, 5 \pmod{8}$ then $2n$ is a $3 \pi/4$-congruent number.
\end{enumerate}
\end{corollary}

\subsection{Cusp forms of weight $3/2$ and Tunnell-like formulae}

We use Waldpurger's theorem to obtain Tunnell-like conditions for a square-free $n$ to be $\pi/4$- or $3\pi/4$-congruent. 

We introduce some notations. Let $k$ be an odd integer $\ge 3$  and $N$ be a positive integer such that $4\mid N$. Let $\chi$ be a quadratic character modulo $N$. For a newform (primitive Hecke eigenform) $F$ of weight $k-1$ and level dividing $N/2$ and trivial character, let $S_{k/2}(N, \chi, F)$ denote the subspace of $S_{k/2}(N,\chi)$ consisting of forms that are {\it Shimura-equivalent} to $F$, i.e., the forms $f$ in $S_{k/2}(N,\chi)$ that are eigenforms under half-integral weight Hecke operators $T_{p^2}$ with the same eigenvalues as $F$ under integral weight Hecke operators $T_p$ for almost all odd primes $p$ coprime to $N$. 
Shimura \cite{Shimura} gave the following direct sum decomposition: 
\begin{equation*}
S_{k/2}(N,\chi)=S_0(N,\chi) \oplus \bigoplus_{F} S_{k/2}(N,\chi,F),
\end{equation*}
where $F$ runs through all newforms of weight $k-1$,
level dividing $N/2$ and character $\chi^2$,  
and
the space $S_0(N,\chi)$ is the subspace spanned by 
single-variable theta-series.
In \cite{Purkait}, the second author gave a refined definition of the Shimura-equivalent spaces and used it to give an algorithm to compute these spaces (see \cite[Theorem 5, Corollary 5.2]{Purkait} for more details). 

Waldspurger's theorem relates the critical value of the $\mathrm{L}$-function of the $n$-th quadratic twist of $F$ to
the $n$-th coefficient of modular forms in $S_{k/2}(N, \chi, F)$. We state a particular version of his theorem \cite[Corollary 2] {Waldspurger} below which is sufficient for our purpose.  

\begin{theorem}[Waldspurger]\label{thm:wald}
Let $F$ be a newform of weight $k-1$, level $M$ and trivial character and let $16\mid M$. Let $f(z) = \sum_{n=1}^{\infty} a_n q^n \in S_{k/2} (N, \chi, F)$ for some $N \geq 1$ 
such that $M$ divides $N/2$ and $\chi$ be a quadratic character modulo $N$. 
Let $\chi_0 (n) := \chi(n) \left( \frac{-1}{n}\right)^{(k-1)/2}$. 
Suppose that $n_1, n_2$ are positive square-free integers such that $n_1/n_2 \in (\mathbb{Q}_{p}^{\times})^{2}$ for all 
$p \mid N$. Then we have the following relation: 
\[ 
a_{n_1}^2 L\left(F \otimes \chi_0^{-1} \left( \frac{n_2}{\cdot}\right), \frac{k-1}{2} \right) \chi(n_2/n_1) n_2^{k/2 - 1} 
= 
a_{n_2}^2 L\left(F \otimes \chi_0^{-1} \left( \frac{n_1}{\cdot}\right), \frac{k-1}{2} \right)  n_1^{k/2 - 1}.
\] 
\end{theorem}


Let $F_1 \in S_2^{\mathrm{new}}(128, \chi_{\mathrm{triv}})$ be the newform corresponding to the elliptic curve $E_1$. It has the following $q$-expansion: 
\[
F_1 = q + 2q^3 + 2q^5 - 4q^7 + q^9 - 2q^{11} + 2q^{13} + 4q^{15} - 2q^{17} + 2q^{19} + O(q^{20}). 
\]
We also consider the newforms $F_{-1}$, $F_{-2}$ corresponding to twists $E_{-1}$ 
and $E_{-2}$ respectively, each of conductor $128$: 
\begin{align*}
F_{-1} &= q - 2q^3 + 2q^5 + 4q^7 + q^9 + 2q^{11} + 2q^{13} - 4q^{15} - 2q^{17} - 2q^{19} + O(q^{20}),\\
F_{-2} &= q + 2q^3 - 2q^5 + 4q^7 + q^9 - 2q^{11} - 2q^{13} - 4q^{15} - 2q^{17} + 2q^{19}  + O(q^{20}).
\end{align*}

We compute the Shimura-equivalent spaces $S_{3/2}(512, \chi_{\mathrm{triv}}, F_1)$, $S_{3/2}(512, \chi_{\mathrm{triv}}, F_{-1})$, $S_{3/2}(512, \left( \frac{2}{n} \right), F_{1})$ and 
$S_{3/2}(512, \chi_{\mathrm{triv}}, F_{-2})$ using the decomposition algorithm \cite[Corollary 5.2]{Purkait}.

We illustrate the case for $S_{3/2}(512, \chi_{\mathrm{triv}}, F_1)$. In this case, using the algorithm it turns out that 
\[
S_{3/2}(512,\chi_{\mathrm{triv}},F_1)=
\left\{
f \in S_{3/2}(512,\chi_{\mathrm{triv}}) \; : \;
T_{3^2} (f) = 2 f, T_{5^2} (f) = 2 f  
\right\};
\]
note that $T_3(F_1)=2F_1,\ T_5(F_1)=2F_1$.
Implementing this in {\tt MAGMA} \cite{Magma}, we get that $S_{3/2}(512,\chi_{\mathrm{triv}},F_1)$ is $2$-dimensional and spanned by $f_1$ and $f_2$ whose $q$-expansions are noted below.  

It turns out that for each of the Shimura-equivalent spaces noted above, we can find a basis that can be expressed in terms of  theta series coming from positive definite ternary quadratic forms. In the below $\theta_{[a,b,c,r,s,t]}$ stands for the theta series of 
quadratic form  $Q=ax^2 +by^2 +cz^2 +ryz+sxz+txy$, that is, 
$\theta_{[a,b,c,r,s,t]} = \displaystyle \sum_{n=0}^{\infty} r(Q,n)q^n$,
where $r(Q,n)$ is the number of representations of $n$ by $Q$.

We obtain the following description of the Shimura-equivalent spaces:
\begin{enumerate}
\item The space $S_{3/2}(512, \chi_{\mathrm{triv}}, F_1)$ has a basis 
$\{ f_1, f_2 \}$, where $f_1$ and $f_2$ are given by the theta series:
\[
f_1 = \frac{1}{2} \left(\theta_{[ 1, 8, 128, 0, 0, 0 ]}- \theta_{[ 4, 8, 33, 0, -4, 0 ]} \right), \qquad 
f_2 = \frac{1}{2} \left(\theta_{[ 3, 5, 19, -2, -2, -2 ]} - \theta_{[ 5, 6, 10, -4, -4, 0 ]} \right)
\]
and 
they have the following $q$-expansions:
\begin{align*}
f_1 & = q + 3q^9 + 2q^{17} + q^{25} + 2q^{33} - 4q^{41} - 3q^{49} + 2q^{57} -
    8q^{65} - 2q^{73} + O(q^{80}), \\
f_2 &= q^3 - 3q^{11} + q^{19} + 2q^{27} + 2q^{35} - q^{43} - 4q^{51} + q^{59} -
    3q^{67} + 3q^{75} + O(q^{80}). 
\end{align*} 

\item The space $S_{3/2}(512, \chi_{\mathrm{triv}}, F_{-1})$ has a basis 
$\{ g_1, g_2 \}$, where $g_1$ and $g_2$ are given by the theta series:
\begin{align*}
g_1 = \frac{1}{2} \left(\theta_{[ 1, 10, 26, -4, 0, 0 ]} - \theta_{[ 4, 9, 10, -8, 0, -4 ]} 
\right), \qquad 
g_2 = \frac{1}{2} \left(\theta_{[ 3, 8, 43, 0, -2, 0 ]} - \theta_{[ 8, 11, 12, -4, 0, 0 ]} \right)
\end{align*} 
and they have the following $q$-expansions:
\begin{align*}
g_1 &= q - q^9 - 2q^{17} + q^{25} - 2q^{33} + 4q^{41} + 5q^{49} - 2q^{57} -
    6q^{73} - q^{81} - 2q^{89} + O(q^{90}), \\
g_2 &= q^3 + q^{11} - 3q^{19} - 2q^{27} + 2q^{35} - q^{43} + q^{59} + q^{67} +
    3q^{75} + q^{83} + O(q^{90}).
\end{align*}

\item The space $S_{3/2}(512, \left( \frac{2}{n} \right), F_{1})$ has a basis 
$\{ h_1, h_2 \}$, where $h_1$ and $h_2$ 
are given by the theta series:
\begin{align*}
h_1 = \frac{1}{2} \left(\theta_{[ 128, 16, 1, 0, 0, 0 ]} - \theta_{[ 33, 16, 4, 0, 4, 0 ]} \right), \qquad 
h_2 = \frac{1}{2} \left(\theta_{[ 47, 28, 7, 4, 6, 20 ]} - \theta_{[ 28, 23, 15, 10, 12, 4 ]} 
\right)
\end{align*}
and they have the following $q$-expansions:
\begin{align*}
h_1 &= q + q^9 + 2q^{17} + 3q^{25} - 2q^{33} - 3q^{49} - 6q^{57} + 4q^{65} -
    2q^{73} - q^{81} + 2q^{89} + O(q^{90}), \\
h_2 &= q^7 - q^{15} - q^{23} + q^{39} - q^{55} + q^{63} - q^{71} + 2q^{79} + q^{87} + O(q^{90}).
\end{align*}
Note that we also have 
\[ h_2 = \frac{1}{4} \left(\theta_{[ 7, 7, 44, -4, -4, -2 ]} - \theta_{[ 12, 15, 15, 14, 4, 4 ]} \right);\] 
this expression will be used in Theorem~\ref{thm:mod16}.

\item The space $S_{3/2}(512, \chi_{\mathrm{triv}}, F_{-2})$ has a basis 
$\{ k_1, k_2 \}$, where $k_1$ and $k_2$ are given by the theta series:
\begin{align*}
k_1 = \frac{1}{4} \left(\theta_{[ 3, 3, 128, 0, 0, -2 ]}-\theta_{[ 4, 8, 35, -8, -4, 0 ]} 
\right), \qquad 
k_2 = \frac{1}{4} \left(\theta_{[ 5, 5, 44, 4, 4, 2 ]} - \theta_{[ 8, 12, 13, -4, -8, 0 ]} \right)
\end{align*}
and have the following $q$-expansions:
\begin{align*}
k_1 &= q^3 + q^{11} + q^{19} + 2q^{27} - 2q^{35} - q^{43} - 3q^{59} + q^{67} -
    q^{75} - 3q^{83} + O(q^{90}), \\
k_2 &= q^5 - q^{13} - q^{29} - q^{37} + q^{45} + q^{53} + q^{61} - 2q^{69} +
    2q^{77} + O(q^{90}).   
\end{align*}

\end{enumerate}

\begin{theorem}\label{thm:Lfunc}
Let $E_1$ be as above and $n$ be a positive 
square-free odd integer. 
\begin{enumerate}
\item
Let $f_1 + \sqrt{2}f_2 := \sum_{n \geq 1} a_n q^n$. Then 
\[
L(E_{-n}, 1) = \begin{cases}
\dfrac{L(E_{-1}, 1)}{\sqrt{n}} a_n^2, &\mbox{if } n \equiv 1, 3 \pmod{8}, \\
0 &\mbox{otherwise}.
\end{cases}
\]
\item
Let $g_1 + \sqrt{2}g_2 := \sum_{n \geq 1} b_n q^n$. Then 
\[
L(E_{n}, 1) = \begin{cases}
\dfrac{L(E_{1}, 1)}{\sqrt{n}} b_n^2, &\mbox{if } n \equiv 1, 3 \pmod{8}, \\
0 &\mbox{otherwise}.
\end{cases}
\]
\item
Let $h_1 + 2\sqrt{2}h_2 := \sum_{n \geq 1} c_n q^n$. Then 
\[
L(E_{-2n}, 1) = \begin{cases}
\dfrac{L(E_{-2}, 1)}{\sqrt{n}} c_n^2 &\mbox{if } n \equiv 1, 7 \pmod{8}, \\
0 &\mbox{otherwise}.
\end{cases}
\]
\item 
Let $k_1 + \sqrt{2}k_2 := \sum_{n \geq 1} d_n q^n$. Then 
\[
L(E_{2n}, 1) = \begin{cases}
\dfrac{L(E_{6}, 1)\sqrt{3}}{\sqrt{n}} d_n^2, &\mbox{if } n \equiv 3, 5 \pmod{8}, \\
0 &\mbox{otherwise}.
\end{cases}
\]
\end{enumerate}
\end{theorem}
\begin{proof}
We prove statement 1, the others can be done similarly. 

Write $f_1= \sum_{n\ge 1}\alpha_nq^n$ and $f_2= \sum_{n\ge 1}\alpha'_nq^n$. 
Applying Theorem~\ref{thm:wald} to $f_1\in S_{3/2}(512, \chi_{\mathrm{triv}}, F_1)$ (so $\chi_0=\left( \frac{-1}{\cdot}\right)$) and taking $n_1=n\equiv 1\pmod{8}$ and $n_2=1$ (so $n_2/n_1 \in (\mathbb{Q}_{2}^{\times})^{2}$), we get that   
\[ 
\alpha_{n}^2 L\left(F_1 \otimes \left( \frac{-1}{\cdot}\right), 1\right) 
= 
\alpha_{1}^2 L\left(F_1 \otimes  \left( \frac{-n}{\cdot}\right), 1 \right)  \sqrt{n},
\] 
i.e., \[\text{for $n\equiv 1\pmod{8},$}\quad \alpha_{n}^2 L\left(E_{-1}, 1\right) 
=  L\left(E_{-n}, 1 \right)  \sqrt{n}.\] This gives statement 1 partially. 

Similarly, applying Theorem~\ref{thm:wald} to $f_2$ and taking $n_1=n\equiv 3\pmod{8}$ and $n_2=3$, we get that 
\[\text{for $n\equiv 3\pmod{8},$} \quad \displaystyle {\alpha'}_{n}^2 
L\left(E_{-3}, 1\right) \sqrt{3} 
=  L\left(E_{-n}, 1 \right)  \sqrt{n}.\] 

By a well-known result of Shimura \cite{Shimura2}, for rational elliptic curves, we know that ${\mathrm{L}(E_{-n},1)}/{\Omega_{E_{-n}}}$,  where $\Omega_{E_{-n}}$ denotes the real period of $E_{-n}$, is rational. Also, using a change of variables argument we can show that $\Omega_{E_{-1}}/\Omega_{E_{-n}}$ equals some rational multiple of $\sqrt{n}$. We use {\tt MAGMA} to compute these rational numbers in particular cases: 
\[\Omega_{E_{-1}}=\sqrt{3}\Omega_{E_{-3}},\quad {\mathrm{L}(E_{-1},1)}/{\Omega_{E_{-1}}}=1/2,\quad  {\mathrm{L}(E_{-3},1)}/{\Omega_{E_{-3}}}=1.\]
Hence, we get the relation 
\[2L\left(E_{-1}, 1\right)= \sqrt{3}L\left(E_{-3}, 1\right).\]

Also, by Proposition \ref{rootnumbers}, $L\left(E_{-n}, 1 \right)=0$ for $n\equiv 5,\ 7 \pmod{8}$.
Combining the above, we get statement 1.

\end{proof}

We have the following consequence of 
the previous theorem.

\begin{corollary}
Assume the Birch and Swinnerton-Dyer conjecture.  
\begin{enumerate}
\item If $n \equiv 1, 3 \pmod{8}$ then 
\begin{align*}
\mathrm{rank }(E_{-n}) &\geq 2 \iff a_n = 0 \quad \mbox{ and } \\ 
\mathrm{rank }(E_{n}) &\geq 2 \iff b_n = 0 .
\end{align*}
\item If $n \equiv 1, 7 \pmod{8}$ then 
\[ \mathrm{rank }(E_{-2n}) \geq 2 \iff c_n = 0. \]
\item If $n \equiv 3, 5 \pmod{8}$ then 
\[ \mathrm{rank }(E_{2n}) \geq 2 \iff d_n = 0. \]
\end{enumerate}
\end{corollary}

\begin{proof}
We prove the first equivalence in statement 1. 
If $n \equiv 1, 3 \pmod{8}$ 
then Proposition \ref{rootnumbers} 
implies that $W(E_n/\mathbb{Q}) = 1$. 
Hence the analytic rank of $E_n$ is even. By the Birch and Swinnerton-Dyer conjecture, the (algebraic) rank of $E_n$ is even. The desired result follows from Theorem \ref{thm:Lfunc} (1).  
The rest can be verified similarly.
\end{proof}

\subsection{Families of non-$\pi/4$ and non-$3\pi/4$-congruent numbers}

For a ternary quadratic form $Q$ and a positive integer $m$, let $r(Q,m)$ denote the number of representations of $m$ by $Q$ (as in the previous subsection) and let $R(Q,m)$ denote the number of essentially distinct primitive representations of $m$ by the genus of ternary quadratic forms containing $Q$. In this section, we consider $r(Q,p)$ for certain families of primes $p$ for quadratic forms that appeared in the last section. To draw any inference on $r(Q,p)$ we need to consider the automorphs of $Q$. We use codes of Gonzalo Tornaria, implemented in {\tt SageMath} \cite{Sage}, to obtain all possible automorphs. 

\begin{theorem}
Let $p$ be a prime with $p \equiv 1 \pmod{8}$. Write $p=a^2 + 8b^2$  and suppose that $b$ is odd. Let $a_p$, $b_p$, and $c_p$ be as in Theorem~\ref{thm:Lfunc}. Then, 
\begin{enumerate}
\item  $a_p \equiv 2 \pmod{4}$; consequently, $L(E_{-p},1) \neq 0$ and so $p$ is not $3\pi/4$-congruent.
\item $b_p \equiv 2 \pmod{4}$; consequently, $L(E_{p},1) \neq 0$ and so $p$ is not $\pi/4$-congruent.
\item $c_p \equiv 2 \pmod{4}$;  consequently, $L(E_{-2p},1) \neq 0$ and so $2p$ is not $3 \pi/4$-congruent.
\end{enumerate}
\end{theorem}

\begin{proof}
\begin{enumerate}
\item 
Let $Q_1 := [ 1, 8, 128, 0, 0, 0 ]$ 
and $Q_2 := [ 4, 8, 33, 0, -4, 0 ]$; these are the forms that occur in $f_1$. 

We first consider $r(Q_1,p)$. Since $Q_1$ is stable under the transformations 
\[ (x,y,z) \mapsto (\pm x,y, z), (\pm x, -y, z), (\pm x, y, -z), \mbox{ and } 
(\pm x, -y, -z), \]
which are all possible transformations of $Q_1$, we have  
\[ \sharp \{ (x,y,z) \in \mathbb{Z}^3 : Q_1(x,y,z)=p \mbox{ and } xyz \neq 0 \}  \equiv 0 \pmod{8}.  \]

Since $\Z[\sqrt{-2}]$ is a principal ideal domain, $a$ and $b$ in the expression for $p$ are uniquely determined up to units, so 
\[\sharp \{ (x,y,z) \in \mathbb{Z}^3 : Q_1(x,y,z) = p \mbox{ and } z = 0 \} 
= 4. \]
Further since $b$ is odd, we cannot write $p$ as $x^2+128z^2$, hence,  
\[
\sharp \{ (x,y,z) \in \mathbb{Z}^3 : Q_1(x,y,z) = p \mbox{ and } xy = 0 \} 
= 0. \]
Thus, $r(Q_1, p) \equiv 4 \pmod{8}$. 

We now look at $r(Q_2,p)$. The possible automorphs of $Q_2$ are given by 
\[ (x,y,z) \mapsto (x,\pm y,z), (-x, \pm y, -z), (x-z, \pm y, -z), (-(x-z), \pm y, z). \]

It is clear that 
$\sharp \{ (x,y,z) \in \mathbb{Z}^3 : Q_2(x,y,z) = p \mbox{ and } z = 0 \} 
= 0$. Also, $p=Q_2(x,y,z)$ with $y=0$ and $xz\ne 0$ is not possible as otherwise, 
$p=4x^2+33z^2-4xz= (2x-z)^2+ 32z^2$, contrary to our assumption. It follows that 
\[ r(Q_2,p)= \sharp \{ (x,y,z) \in \mathbb{Z}^3 : Q_2(x,y,z)=p \mbox{ and } yz \neq 0 \}  \equiv 0 \pmod{8}.  \]
Consequently, 
$a_p = \frac{1}{2}(r(Q_1, p) - r(Q_2, p)) \equiv 2 \pmod{4}$. 

\item Let $Q_3 := [ 1, 10, 26, -4, 0, 0 ]$ 
and $Q_4 := [ 4, 9, 10, -8, 0, -4 ]$; 
these are the forms that occur in $g_1$.
Observe that if $p \equiv 1 \pmod{8}$ and $p = Q_3(x,y,z)$, then $x$ is odd and 
$y \equiv z \pmod{2}$. Put $X=x$, $Y = \frac{y-z}{2}$ and $Z = \frac{y + z}{2}$.
Then $p = Q_3(x,y,z) = Q'_3(X,Y,Z)$, where 
$Q'_3 = [1,40,32,-32,0,0]$.
The possible automorphs of $Q'_3$ are given by 
\[ (X,Y,Z) \mapsto  \pm (X,Y,Z), \pm(X,-Y,-Z), \pm(-X,Y, Y - Z) \mbox{ and }  \pm(X,Y,Y-Z).  \]
Hence,
\[r(Q_3, p) = r(Q'_3, p) = \sharp \{ (x,y,z) \in \mathbb{Z}^3 : Q_3'(x,y,z) = p \mbox{ and } xy \ne 0 \} \equiv 0 \pmod{8}; \]
in the above, $x$ must be non-zero as $x$ is odd and $y$ must be non-zero as otherwise $p=x^2+32z^2$, contradiction to our assumption on $p$.  

We now consider $Q_4$. Clearly,  
$\sharp \{ (x,y,z) \in \mathbb{Z}^3 : Q_4(x,y,z) = p \mbox{ and $y$ is even} \} 
= 0$.
Suppose that $Q_4(x,y,z) = p$ with $y$ odd. If $z$ is odd, then $p = Q_4(x,y,z) \equiv 3 \pmod{8}$, which is absurd, so $z$ must be even. Put $X=x$, $Y=y$, and $z = 2Z$. Then 
$p =Q_4(x,y,z) = Q'_4(X,Y,Z)$, where 
$Q'_4 = [4,9,40,-16,0,-4]$. 
The possible automorphs of $Q'_4$ are given by 
\begin{small}
\[ 
(X,Y,Z) \mapsto \pm(X,Y,Z), \pm(X-Y,-Y,-Z), \pm(-X+Z,-Y+2Z,Z), \mbox{ and } \pm(-X+Y-Z,Y-2Z,-Z).
\]
\end{small}
Also, since $Q'_4(X,Y,0) = Q_4(x,y,0) = (2x-y)^2 + 8y^2$, the hypothesis implies that  
\[
\sharp \{ (x,y,z) \in \mathbb{Z}^3 : Q_4(x,y,0) = p \} 
= 4.
\]
Therefore, $r(Q_4, p) \equiv 4 \pmod{8}$. Consequently, 
$b_p = \frac{1}{2}(r(Q_3, p) - r(Q_4, p)) \equiv 2 \pmod{4}$. 

\item Let $Q_5 := [ 128, 16, 1, 0, 0, 0 ]$ 
and $Q_6 := [ 33, 16, 4, 0, 4, 0 ]$; these are  
the forms that occur in $h_1$.
The possible automorphs of $Q_6$ are given by 
\[ (x,y,z) \mapsto \pm (x,y,z), \pm(-x,y,-z), \pm(x,-y,-x-z), \mbox{ and } \pm(-x,-y,x+z). \] 
Also, by the assumption on $p$, it cannot be written as $Q_6(x,y,z)$ with $xy=0$. Thus, 
\[r(Q_6, p) = \sharp \{ (x,y,z) \in \mathbb{Z}^3 : Q_6(x,y,z) = p \mbox{ and } xy \ne 0 \} \equiv 0 \pmod{8}. \]

Next we consider $Q_5$. The possible automorphs of $Q_5$ are given by
\[ (x,y,z) \mapsto \pm (x,y,z), \pm(x,-y,-z), \pm(-x,y,-z), \mbox{ and } \pm(-x,-y,z), \]
so  \[ \sharp \{ (x,y,z) \in \mathbb{Z}^3 : Q_5(x,y,z) = p \mbox{ and } xyz \ne 0 \} \equiv 0 \pmod{8}.\]
Also, since $p\equiv 1 \pmod{8}$, there exist (unique up to units) integers $r$, $s$ such that $p=s^2+16r^2=Q_5(0,r,s)$, so \[ \sharp \{ (x,y,z) \in \mathbb{Z}^3 : Q_5(x,y,z) = p \mbox{ and } x=0 \} =4.\]
Finally by the assumption on $p$, 
\[ \sharp \{ (x,y,z) \in \mathbb{Z}^3 : Q_5(x,y,z) = p \mbox{ and } yz =0\} =0.\]
It follows that $r(Q_5,p) \equiv 4 \pmod{8}$.
Therefore, 
$c_p = \frac{1}{2}(r(Q_5, p) - r(Q_6, p)) \equiv 2 \pmod{4}$. 
\end{enumerate}
\end{proof}

\begin{theorem}
Let $p$ be a prime with $p \equiv 3 \pmod{8}$. 
Then $b_p$ is odd. 
Thus, $L(E_{p},1) \neq 0$ and so $p$ is not $\pi/4$-congruent.
\end{theorem}

\begin{proof}
Let $Q_7 := [ 3, 8, 43, 0, -2, 0 ]$ and 
$Q_8 := [ 8, 11, 12, -4, 0, 0 ]$, these are 
the forms that occur in $g_2$. 
The possible automorphs of $Q_7$ are given by 
\[ 
(x,y,z) \mapsto \pm(x,y,z) \mbox{ and } \pm(-x,y,-z), 
\]
while that of $Q_8$ are given by 
\[ 
(x,y,z) \mapsto \pm(x,y,z) \mbox{ and } \pm(x,-y,-z). 
\]
It follows that 
\begin{align*}
\sharp \{ (x,y,z) \in \mathbb{Z}^3 : Q_7(x,y,z) &= p \mbox{ and } y \neq 0 \} 
\equiv 0 \pmod{4}, \mbox{ and } \\
\sharp \{ (x,y,z) \in \mathbb{Z}^3 : Q_8(x,y,z) &= p \mbox{ and } x \neq 0 \} 
\equiv 0 \pmod{4}. 
\end{align*}
Let $x$, $z$ be integers such that 
$p = Q_7(x,0,z)$. Letting $r = x+z$ and $s = -4z$, we find that $p = 3r^2 + 2rs + 3s^2$, with $4 \mid s$. 
On the other hand, suppose $Y$, $Z$ are integers 
such that $p = Q_8(0,Y,Z)$. Then $Y$ must be odd, and letting 
$u = 2Z - Y$ and $v = 2Y$, we find that 
$p = 3u^2 + 2uv + 3v^2$, with $2 || v$.  
Now, $p \equiv 3 \pmod{8}$ can be uniquely (up to units) written as $p=x^2+2y^2$ with $x$, $y$ odd. Writing $x=a-b$, $y=a+b$, we see that $p$ 
has a unique (up to sign) representation 
as $p = 3a^2 + 2ab + 3b^2$ with $ab$ even. Thus, 
\begin{align*}
\sharp \{ (x,0,z) \in \mathbb{Z}^3 : Q_7(x,0,z) &= p \} 
\equiv 
\begin{cases}
0 \pmod{4}, \mbox{ if } 2 || ab  \\
2 \pmod{4}, \mbox{ if }4 \mid ab, 
\end{cases} 
\mbox{ and } \\
\sharp \{ (0,y,z) \in \mathbb{Z}^3 : Q_8(0,y,z) &= p \} 
\equiv 
\begin{cases}
2 \pmod{4}, \mbox{ if } 2 || ab  \\
0 \pmod{4}, \mbox{ if } 4 \mid ab. 
\end{cases} 
\end{align*}
Combining these with the observation above gives the desired result.
\end{proof}

\begin{theorem}\label{thm:mod16}
Let $p$ be a prime with $p \equiv 7 \pmod{16}$. 
Then the $p$-th coefficient of $h_2$ is odd. 
Thus, $L(E_{-2p},1) \neq 0$ and so $2p$ is not $3 \pi/4$-congruent.
\end{theorem}

\begin{proof}
Let $Q_9 := [ 7, 7, 44, -4, -4, -2 ]$ and 
$Q_{10} := [ 12, 15, 15, 14, 4, 4 ]$; these are the forms that appear in $h_2$. 
The possible automorphs of $Q_9$ are given by 
\[ 
(x,y,z) \mapsto \pm(x,y,z) \mbox{ and } \pm(-y,-x,-z),  
\]
while that of $Q_{10}$ are given by
\[ 
(x,y,z) \mapsto \pm(x,y,z) \mbox{ and } \pm(-x,-z,-y).  
\]
It follows that both $r(Q_{9},p)$ and $r(Q_{10},p)$ are divisible by $4$. Moreover, the 
classes of $Q_9$ and $Q_{10}$ are the 
only members of the genus of ternary 
quadratic forms that contain $Q_9$.  
Hence, we get 
\begin{equation}\label{for_h2}
R(Q_9, p) = \frac{r(Q_{9},p)}{4} + \frac{r(Q_{10},p)}{4}.
\end{equation}
By Theorem 86 of \cite{Jones}, we have 
\[ 
R(Q_9, p) = 2^{-t(d/\Omega^2)} h(-8p) \rho, 
\]
where $t(w)$ is the number of odd prime 
factors of $w$, 
$d$ is the determinant of the matrix $A$ of $Q_9$, 
$\Omega$ is the greatest common divisor of 
two-rowed minor determinants of $A$ and 
$\rho$ depends on the discriminant. 
In our case, we have $d= 2^{11}$, $\Omega = 2^4$ and $\rho = 1/4$. So 
$R(Q_9, p) = \frac{1}{4} h(-8p)$. 
Combining with equation (\ref{for_h2}), 
we obtain 
$r(Q_9,p) + r(Q_{10}, p) = h(-8p)$. 
Using Proposition 2 of \cite{Pizer}, 
our hypothesis on $p$ implies $h(-8p) \equiv 4 \pmod{8}$. 
Therefore, 
\[
4 c'_p = r(Q_9, p) - r(Q_{10},p) = 
h(-8p) - 2 r(Q_{10},p) \equiv 4 \pmod{8},
\] 
where $c'_p$ is the $p$-th coefficient of $h_2$. The result follows.
\end{proof}

\section{The distribution of $\pi/4$- and $3 \pi/4$-congruent numbers}\label{distribution}

Our aim in this section is to show that if $m$ is a positive integer and $a$ is a non-negative integer, then 
there are infinitely many $\pi/4$-congruent (resp.\ $3 \pi/4$-congruent) numbers of the form $a + md$, where $d$ is an 
integer. Its proof is based on the following 
observation.

\begin{proposition}\label{pi/4-family}
A square-free positive integer $n>1$ is $\pi/4$-congruent (resp.\ $3 \pi/4$-congruent) 
if and only if $n$ is the square-free part of 
\begin{equation}\label{pi/4-formula}
rs(r^2 + 2rs - s^2)
\qquad \mbox{ (resp.\ $rs(r^2 - 2rs - s^2)$)},
\end{equation} 
where $r$ and $s$ are relatively prime positive integers.
\end{proposition}

\begin{proof}
We only give the proof for the case of  
$\pi/4$-congruent number as 
the proof for the other case  
is similar.

The elliptic curve $E_n': ny^2 = x^3 + 2x^2 - x$ 
is $\mathbb{Q}$-isomorphic to $E_n$, 
with the isomorphism given by $(x,y) \mapsto (nx, n^2y)$, for a point $(x,y)$ in $E_n'$. 
By Lemma \ref{ellcurve_pi}, 
$n$ is $\pi/4$-congruent if and only if 
$E_n'$ has a rational point 
$P = (x,y)$ of order greater than $2$. 
Replacing $P$ by $P + (0,0)$, if necessary, we may assume that $x>0$. 
Write $x = r/s$ with relatively prime 
positive integers $r$ and $s$. Then 
the equation 
$ny^2 = \left(\frac{r}{s} \right)^3 + 2 \left(\frac{r}{s} \right)^2 - \left(\frac{r}{s} \right)$ gives 
\[ n(ys^2)^2 = rs(r^2 + 2rs - s^2). \]
Therefore, 
$n \equiv rs(r^2 + 2rs - s^2) \pmod {(\mathbb{Q}^{\times})^2}$.

\end{proof}

\begin{theorem}
Let $m > 1$ be a positive integer and $a$ 
be a non-negative integer. 
Then the congruence class of $a \pmod{m}$
contains infinitely many $\pi/4$-congruent numbers which are  
inequivalent modulo squares of rational numbers.
A similar statement holds for $3\pi/4$-congruent numbers. 
\end{theorem}

\begin{proof}
We prove the statement for $\pi/4$-congruent numbers; proof for 
$3\pi/4$-congruent numbers is analogous.
Define the set 
$U = \{
u \in \mathbb{N} : u \equiv a \pmod{m}
\}$. 
For each $u\in U$, set 
\[ r_u= u m^2 + 1 \mbox{ and } s_u = u m^2. \]
By Proposition \ref{pi/4-family}, since $r_u$ and $s_u$ are relatively prime, substituting these values into  (\ref{pi/4-formula}),   we get the $\pi/4$-congruent number
$\tilde{n}_{u} = u m^2(u m^2 + 1)(2u^2 m^4 + 4 u m^2 + 1) > 1$. 
We may scale by $m^2$ to obtain the 
$\pi/4$-congruent number 
\[
n_u := u(u m^2 + 1)(2u^2 m^4 + 4u m^2 + 1).
\]
Observe that 
\[ n_u \equiv u \equiv a \pmod{m}, \qquad \mbox{ for each }  u \in U. \]
Suppose  there are only finitely many inequivalent such $n_u$'s modulo $(\mathbb{Q}^{\times})^2$, say, $\{n_{u_1}, n_{u_2},\ldots, n_{u_k}\}$. Since $U$ is infinite, there must exist some  $1\le i\le k$ for which there exists an infinite subset $S$ of $U$ such that $n_s \equiv n_{u_i} \pmod{(\Q^{\times})^2}$ for each $s\in S$, i.e., $n_{u_i}n_s$ is an integer square for each $s\in S$. Thus the algebraic curve 
\[ C_{u_i} :  y^2 = n_{u_i} x(xm^2  + 1)(2x^2m^4 + 4 xm^2  + 1) \]
has infinitely many integer solutions given by $(x,y) = (s, y_s)$, where $s \in S$ and $y_s \in \mathbb{Z}$ such that $n_{u_i} n_s = y_s^2$. But this contradicts Siegel's theorem \cite{Siegel}. 
Thus, there must be infinitely many inequivalent $n_u$'s modulo 
$(\mathbb{Q}^{\times})^2$. 

\end{proof}

The limitation of the previous theorem is that 
it does not say anything about square-free values of 
the binary form $f(r,s) = rs(r^2 \pm 2rs - s^2)$
in a residue class modulo $m$. For most residue classes, 
this can be remedied using asymptotic estimates for the distribution of 
square-free values of binary forms by Gouv{\^e}a and Mazur \cite{Gouvea-Mazur} and 
Greaves \cite{Greaves}, complemented with the result by 
Stewart and Top \cite{SteTop}. 

In the following discussion, 
we let $A$, $B$ and $M$ be integers with $M \geq 1$. 
Let $F(u,v)$ be a homogeneous binary form 
of degree $d>2$ with integer coefficients. 
Let $w$ be the largest positive integer such that $w^2$ divides $F(u,v)$ for all integers $u$, $v$ with $u \equiv A \pmod{M}$, $v \equiv B \pmod{M}$. 
We consider two counting 
functions associated with square-free values of $F$ and 
recall the asymptotic estimates indicated above. 
We state versions of these results that apply to our situation.

Let $x$ be a real number. Let $N_2(x)$ denote the number of 
pairs of integers $u, v$ with 
\[ \begin{cases}
1 \leq u \leq x; \quad &u \equiv A \pmod{M}, \\
1 \leq v \leq x; \quad &v \equiv B \pmod{M}
\end{cases} 
\]
for which $F(u,v)$ is square-free. 
Let $R_2(x)$ denote the number of square-free integers $t$ with $|t| \leq x$ 
for which there exist $u \equiv A \pmod{M}$ and $v \equiv B \pmod{M}$ such that $F(u,v)=tw^2$. 

\begin{theorem}[Gouv{\^e}a-Mazur \cite{Gouvea-Mazur}, Greaves \cite{Greaves}]\label{Gouvea-Mazur-Greaves}
Let $F(u,v)$ be an integral binary form  
of degree $d$ with non-zero discriminant having no non-trivial fixed square divisors.
Suppose that all the irreducible factors 
of $F$ over $\mathbb{Q}$ are of degree $\leq 5$.  
Then 
\[N_2(x) = Cx^2 + \mathcal{O} \left(\frac{x^2}{\log x} \right), \]
where the 
$\mathcal{O}$-constant depends only on $F$ and $C$ is defined as follows:
\begin{equation}\label{GMconstant}
C := \frac{1}{M^2} \prod_{p : \text{prime}} \left( 1 - \frac{r(p^2)}{p^4} \right), 
\end{equation}
where $r(p^2) := (\gcd(p^2, M))^2 \rho(p^2)$ and 
$\rho(p^2)$ denotes the number of 
noncongruent modulo $p^2$ solutions of $F(u,v) \equiv 0 \pmod{p^2}$ with 
$u,v  \equiv A, B \pmod{ \gcd(p^2, M)}$. 
\end{theorem}

In the above theorem $F(u,v)$ having no non-trivial fixed square divisors means that if $w_0$ is the largest positive integer such that $w_0^2$ divides $F(u,v)$ for all integers $u,\ v$ then $w_0=1$. Thus, it follows from the definition of $C$ that $C>0$ if $M=1$. If $M>1$ then $C$ may vanish. Gouv{\^e}a and Mazur (\cite{Gouvea-Mazur}, Proposition 5) 
gave a criterion for non-vanishing of this constant under the condition that the irreducible factors of $F(u,v)$ have degree $\leq 3$.  For the situation of interest to us (Theorem~\ref{density_pi/4}), 
we will verify that this constant is always positive.


\begin{theorem}\label{density_pi/4}
Let $m > 1$ be a positive integer and $a$ 
be an integer such that $\gcd(a,m)$ is square-free.  
Suppose $\gamma \in \{ \pi/4, 3\pi/4 \}$. 
Then the congruence class of $a \pmod{m}$
contains infinitely many square-free 
$\gamma$-congruent numbers. 
\end{theorem}

\begin{proof} 
We prove the result for $\gamma = \pi/4$.  
The proof for the case $\gamma = 3\pi/4$ 
is similar.
Consider the binary form 
\[ f(r,s) = rs(r^2 + 2rs - s^2) \]
and set $r = um^2$ and $s=v$. Then 
\[f(r,s) =  uv m^2 (u^2 m^4 + 2uv m^2 - v^2). \]
Thus, the square-free values of the binary form \[F(u,v) = uv (u^2 m^4 + 2uv m^2 - v^2)\] gives square-free parts of $f(r,s)$ and hence square-free $\gamma$-congruent numbers. 
 We apply Theorem~\ref{Gouvea-Mazur-Greaves} to the above $F$ with $M=m$, $A = a$ and $B = -1$.
If $u \equiv a \pmod{m}$ and $v \equiv -1 \pmod{m}$ 
then $F(u,v) \equiv u(-v)^3 \equiv a \pmod{m}$.
We claim that $F$ achieves square-free values 
infinitely often for $u \equiv a \pmod{m}$ and $v \equiv -1 \pmod{m}$. 

Observe that $F$ has non-zero discriminant, has no non-trivial fixed square divisors and all its irreducible factors over $\Q$ are of degree $\le 2$, so by  
Theorem \ref{Gouvea-Mazur-Greaves}, 
\[N_2(x) = Cx^2 + \mathcal{O} \left(\frac{x^2}{\log x} \right), \]
where the constant $C$ is given by (\ref{GMconstant}). 
We now show that $C$ is positive by applying Gouv{\^e}a-Mazur's criterion (\cite{Gouvea-Mazur}, Proposition 5). 

Put $C_p := 1 - r(p^2)/p^4$.
Note that no prime divides all the coefficients of $F(u,v)$, 
since the coefficient of $u v^3$ in $F(u,v)$ is $-1$. 
This implies that we are not in case (2) of the said proposition. If $p$ 
does not divide $m$, then by statement (5) of the proposition, $C_p$ is non-zero. Suppose $p$ divides $m$. 
We now appeal to statement (3) of the proposition. 
If $p^2$ divides $m$ then, since $\gcd(a,m)$ is square-free, 
$F(a, -1) \equiv a \not\equiv 0 \pmod{p^2}$, consequently  $C_p \neq 0$ by statement (3)-(i).  
Suppose $\text{ord}_p(m)=1$. 
The partial derivative with respect to $u$ of $F$ 
modulo $p$, given by 
\[ \frac{\partial F}{\partial u}(u,v) = - v^3 \pmod{p}, \]
does not vanish at $(u,v) = (a, -1)$. Thus, 
the point $(a, -1)$ does not represent a singular point 
of $F(u,v) \equiv 0 \pmod{p}$ and so by statement (3)-(ii), $C_p \neq 0$. By statement (1) of the proposition, it follows that $C$ is positive.

Consequently, the number of pairs $u,v$ with 
$(u,v) \equiv (a,-1) \pmod{m}$ for which $F(u,v)$ is 
square-free has positive density.  
In particular, the number of square-free values of 
$F(u,v)$ in the residue class $a \pmod{m}$ is infinite.
So this residue class contains infinitely many 
$\pi/4$-congruent numbers. 

\end{proof}

\section{Two subfamilies of $E_n$ of rank $\ge 2$}\label{rankatleasttwo}

Mestre \cite{Mestre} proved that for a given elliptic curve there are infinitely many quadratic twists of rank at least $2$, and  Stewart and Top quantified this result in ~\cite[Theorem 3]{SteTop}. In this section, we give two families of $E_n$ of rank $\ge 2$. 

The first family essentially comes from the work of Stewart and Top. By doing a suitable transformation $E_1$ can be given by its short Weierstrass form: $J: y^2 = x^3 - 3024x + 58752$. Note that $3024$ can be expressed by the quadratic form $a^2+ac+c^2$ with $a=60$ and $c=-12$. We now use ~\cite[Theorem 4]{SteTop}. Define 
\[d_t= 2^6 3^3 (t^2 - 8t - 2)(t^2 + t + 1)(2t^2 - 4t - 7)(7t^2 + 10t + 1). \]
Then $J_{d_t}: d_ty^2 = x^3 - 3024x + 58752$ has independent points $M_t:=\left(\frac{12t^2 + 120t + 48}{t^2+t+1}, \frac{1}{(t^2+t+1)^2} \right)$, $N_t:=\left(\frac{48t^2 - 24t - 60}{t^2+t+1}, \frac{1}{(t^2+t+1)^2} \right)$ over $\Q(t)$. It follows from Silverman's specialization theorem (\cite[Lemma 1]{SteTop}) that $J_{d_t}$ has rank at least two for all but finitely many rational numbers $t$, while the result of Stewart and Top implies that for $N$ large enough, the number of square free integers $d$ with $|d|\le N$ such that $J_d$, and hence $E_d$, has rank at least $2$ is at least $\mathcal{O}(N^{1/4})$.
Observe that $d_t$ is positive for almost all integer values of $t$, so there are infinitely many $\pi/4$-congruent numbers $n$ such that $E_n$  has rank greater than or equal to  $2$.

We now look at another subfamily by considering $E_1$ in its original form. 
Recall that for a positive integer $n$ of the form $n=rs(r^2+2rs-s^2)$, where $r, s$ are positive integers, the curve $E_n':ny^2=x^3+2x^2-x$ has a rational (non-torsion) point $(\frac{r}{s}, \frac{1}{s^2})$  
and so the isomorphic curve $E_n$ has a non-torsion rational point.

Let $t$ be an integer such that $t \ne 1$. Set $r=t^2+1$, $s=2$ and define 
\[n_t= rs(r^2+2rs-s^2)= 2(t^2+1)(t^4+6t^2+1).\] 
So the curve $E_{n_t}'$ has a point $P_t:=\left(\frac{t^2+1}{2}, \frac{1}{4} \right)$, and the isomorphic curve $E_{n_tt^2}'$ has a point $\left(\frac{t^2+1}{2t^2}, \frac{1}{4t^4} \right)$ giving another point $Q_t:=\left(\frac{t^2+1}{2t^2}, \frac{1}{4t^3} \right)$ on $E_{n_t}'$. 
Observe that $P_t$ and $Q_t$ are non-torsion points as the torsion subgroup of $E_{n_t}'$ is $\Z/2\Z$. 


Thanks to the specialization theorem, to show independence of $P_t$ and $Q_t$, it is enough to show 
independence of the two points for a particular rational value of $t$. For $t=2$, $n_t=410$ and 
\[E_{n_t}: y^2 = x^3 + 820x^2 - 168100x. \]
The points corresponding to $P_t$, $Q_t$ are 
$(1025, 42025)$, $(1025/4, 42025/8)$ respectively and we verify using {\tt MAGMA} \cite{Magma} that these points are linearly independent. 
Thus, $E_{n_t}$ has rank at least two for all but finitely many rational numbers $t$.

Consider the quadratic form 
$F(u,v)=2(u^2+v^2)(u^4+6u^2v^2+v^4)$.  For $t=\frac{u}{v}$, the numbers $n_t$ and $F(u,v)$ are the same up to rational squares and so give the same set of $\pi/4$-congruent numbers. 
Note that $F(u,v)$ is a degree $6$ form with non-zero discriminant and all its irreducible factors have degree $\le 4$. It is easy to see that $F(u,v)$ has no non-trivial square divisors. 
Thus, by Theorem~\ref{Gouvea-Mazur-Greaves}, the number of pairs $(u,v)$ such that $F(u,v)$ is square-free has a positive density. 

Preliminary computations show that $F(u,v)$ (and its square-free part) does not belong to all the congruence class modulo $m$ for a given integer $m>1$. 
Nevertheless, we can apply Stewart-Top~\cite{SteTop} to the quadratic form $F$ to obtain the following statement.
\begin{corollary}
Let $A, B, m$ be integers with $m>1$. There exists a positive constant $\kappa$ depending on $m$ and $F$ such that for sufficiently large positive integer $N$, the number of square free integers in $[-N, N]$ of the form $F(u,v)$ (up to some fixed square $w^2$) with $u\equiv A \pmod{m}$, $v\equiv B \pmod{m}$ is greater than $\kappa N^{1/3}$.
Consequently, there are infinitely many $\pi/4$-congruent numbers $n$ of the form $F(u,v)$ with $u,v$ as above such that $\mathrm{rank}(E_n)\ge 2$.
\end{corollary}

{\bf Remark.} Given $A$, $B$ and $m$, a similar statement holds for $\pi/4$-congruent numbers of the form $d_t$, though in this case as we observed earlier, the order is at least $\mathcal{O}(N^{1/4})$. 

\section{Tilings of the square with congruent right triangles with condition on the perpendicular sides}\label{tiling}

A tiling (or dissection) of a polygon 
$A$ is a decomposition of $A$ into 
pairwise non-overlapping polygons. 
Recently, Laczkovich established a connection between the 
dissection of an equilateral triangle with rational sides  
into congruent triangles and the 
$\pi/3$- and $2 \pi /3$-congruent number 
elliptic curves. 

\begin{theorem}[\cite{Lacz2}, Theorem 1.1]
For every positive and square-free integer $n$ the following are equivalent.
\begin{enumerate}
\item There is a positive integer $k$ 
such that an equilateral triangle with rational sides can 
be dissected into $nk^2$ congruent triangles.
\item Either $n \leq 3$, or 
at least one of the elliptic curves $y^2 = x(x-n)(x+3n)$ or $y^2 = x(x+n)(x-3n)$ has a rational point with $y\ne 0$.
\end{enumerate}
\end{theorem}
Note that the condition $y^2 = x(x-n)(x+3n)$ or $y^2 = x(x+n)(x-3n)$ has a rational point with $y\ne 0$ is the same as $n$ is $\pi/3$- or $2 \pi /3$-congruent \cite{Yoshida}. 

We consider squares with rational sides and relate the study of $\pi/4$-congruent 
and $3 \pi/4$-congruent numbers with 
tilings of such squares  
by congruent triangles. 
It is known (\cite{Lacz1}, Corollary 3.7) 
that in every tiling of such squares  
with congruent triangles, the pieces 
must be right triangles with commensurable perpendicular sides.  
In this note, we call a right triangle  \emph{almost rational} 
if the perpendicular sides $a$ and $b$  
are rational numbers such that $a^2 + 2b^2 \pm 2ab \in \mathbb{Q}^2$. 

\begin{proposition}
Let $n$ be a positive square-free odd integer. The following are equivalent.
\begin{enumerate}
\item There is a positive integer $k$ 
such that a square with rational sides 
can be dissected 
into $2nk^2$ congruent almost rational right triangles. 
\item $2n$ is $\pi/4$-congruent or 
$3 \pi /4$-congruent.
\end{enumerate}
\end{proposition}

\begin{proof}
Any square can be obtained from the unit square 
by performing an appropriate shrinking or 
stretching transformation, and vice versa. Consequently, 
a square with rational sides has a tiling with $2nk^2$ congruent 
almost rational right triangles 
if and only if the unit square 
has a tiling with $2nk^2$ congruent 
almost rational right triangles.
We assume henceforth that the square is the unit square.

If the unit square can be dissected 
into $2nk^2$ congruent almost rational right triangles 
with rational sides $a,b$ then  
there exists a rational number $c$ such that 
$a^2 + 2b^2 \pm 2ab = c^2$. 
Equivalently we have 
$(ak/2)^2 + 2(bk/2)^2 \pm 2(ak/2)(bk/2) = (ck/2)^2$.
Comparing the areas of the square and the tiles we have  
$1 = abnk^2 = 4n(ak/2)(bk/2)$. 
Thus, 
$2n$ is $\pi/4$-congruent or 
$3 \pi/4$-congruent.

Conversely, suppose $2n$ is $\pi/4$-congruent or 
$3 \pi/4$-congruent.
Let $a, b$ be positive rational numbers such that 
$a^2 + 2b^2 \pm 2ab \in \mathbb{Q}^2$ 
and $4abn = 1$. 
Clearly, we have $ab < 1/4$. 
Choose an integer $t$ large enough 
so that the square can be dissected into 
congruent rectangles of sides $a/t$ and $b/t$.
One of the diagonals will divide each of 
these small rectangles into two 
congruent right triangles. 
Since $(a/t)^2 + 2(b/t)^2 \pm 2ab/t^2$ 
is still a square of a rational number, 
these right triangles are almost rational. 
If $2N$ is the number of pieces of 
these triangles in the decomposition for 
the square then comparing areas, we find that 
\[
N= \frac{t^2}{ab} = 4nt^2.
\] 
Thus, we are able to obtain a tiling of the square into $8nt^2$ congruent almost rational right triangles.
\end{proof}

Corollary \ref{BSD-root} claims that 
under the Birch and Swinnerton-Dyer conjecture, 
an even integer is either 
$\pi/4$-congruent or $3 \pi/4$-congruent. 
Thus, we have the following result.

\begin{corollary}
Let $n$ be a positive square-free odd integer. 
If the Birch and Swinnerton-Dyer conjecture holds, 
then for a given square with rational sides there is a positive integer $k$ 
such that the square can be dissected 
into $2nk^2$ congruent almost rational right triangles. 
\end{corollary}


{\small
\begin{thebibliography}{100}
	
\bibitem{Magma}
	W. Bosma, J. Cannon and C. Playoust,  
	The Magma algebra system. I. The user language, J. Symbolic Comput. {\bf 24} (1997), 235--265.

\bibitem{Fujiwara}
	M. Fujiwara,
	$\theta$-congruent numbers, in: Number Theory, 
	K. Gy\"ory, A. Peth\"o, and 
	V. S\'os (eds.), de Gruyter, (1997), 235--241.

\bibitem{Gouvea-Mazur}
	F. Gouv{\^e}a and B. Mazur,
	The square-free sieve and the rank of elliptic curves. 
	J. Amer. Math. Soc.
	{\bf 4} (1991), 1--23.
	
\bibitem{Greaves}	
	G. Greaves,
	Power-free values of binary forms. 
	Quart. J. Math. Oxford Ser. (2) 
	{\bf 43} (1992), 45--65.	

	
\bibitem{Jones}
	B. Jones, 
	The Arithmetic Theory of Quadratic Forms. 
	Carus Math. Monographs 10, 
	Math. Assoc. Amer., Wiley, 1950.	

\bibitem{Kan}
	M. Kan,
	$\theta$-congruent numbers and elliptic curves.
	Acta Arith., XCIV.2 (2000), 153--160.	
	
	
\bibitem{Lacz1}
	M. Laczkovich,
	Tilings of convex polygons with congruent triangles. 
	Discrete Comput. Goem. {\bf 48}-2 (2012), 330--372.	
	
\bibitem{Lacz2}
	M. Laczkovich,
	Rational Points of Some Elliptic Curves Related to the Tilings of the Equilateral Triangle.	
	Discrete Comput. Geom. (2020) 64:985--994.

\bibitem{Lario}
	J-C. Lario,
	Al-Karaj{\'i} y yo.
	Gac. R. Soc. Mat. Esp. 
	{\bf 19 }:1 (2016), 133--149.

\bibitem{Mestre} 
    J.-F. Mestre, 
    Rang des courbes elliptiques d'invariant donne.
    C. R. Acad. Sci. Paris Ser. I 314 (1992), 919-922.
	
\bibitem{Pizer}
	A. Pizer,
	On the 2-Part of the Class Number of 
	Imaginary Quadratic Number Fields. 
	J. Number Theory {\bf 8} (1976), 
	184--192.
	
\bibitem{Purkait} S. Purkait,  On Shimura's decomposition. 
Int. J. Number Theory {\bf 9} (2013), 1431--1445.




\bibitem{Rizzo} O. G. Rizzo, Average root numbers for a nonconstant family of elliptic Curves. Compos. Math. {\bf 136} (2003), 1--23.

\bibitem{SteTop}
	C. L. Stewart and J. Top,
	On ranks of twists of elliptic curves and power-free values of binary forms. 
	J. Amer. Math. Soc. {\bf 8} (1995), 943--973.	

\bibitem{SZ}
	J. Shu and S. Zhai,
	Generalized Birch lemma and the 2-part of the Birch and 
	Swinnerton-Dyer conjecture for certain elliptic curves. 
	To appear in: J. Reine Angew. Math. (2021). doi: 10.1515/crelle-2021-0004.
		
\bibitem{ST}		
	R. Stroeker and J. Top,
	On the equation $Y^2 = (X+p)(X^2 + p^2)$. 
	Rocky Mount. J. Math. {\bf 24} (1994), 1135--1161.

\bibitem{Sage}
	The Sage Developers. 
	SageMath, the Sage Mathematics Software System (Version 9.0), sagemath.org (2020).	

\bibitem{TopYui}
	J. Top and N. Yui, 
	Congruent number problems and their variants. In: \emph{Algorithmic number theory: lattices, number fields, curves and cryptography}. 
	Vol. 44. Math. Sci. Res. Inst. Publ. Cambridge Univ. Press, Cambridge, 2008, 613--639.

\bibitem{Waldspurger} J. -L. Waldspurger, 
     Sur les coefficients de Fourier des formes modulaires de poids demi-entier. 
     J. Math. Pures Appl. {\bf 60} (1981), 375--484.	

\bibitem{Shimura} G.\ Shimura,
The critical values of certain
zeta functions associated with modular forms of half-integral weight. 
J. Math. Soc. Japan {\bf 33} (1981), 649--672.

\bibitem{Shimura2} G.\ Shimura, On the periods of modular forms. Math. Ann. {\bf 229} (1977), 211--221.

\bibitem{Siegel}
	C. L. Siegel, 
	{\"U}ber einige Anwendungen diophantischer Approximationen. 
	Abh. Preuss. Akad. Der Wissenschaften Phys.-math. Kl. {\bf 1} (1929), 209--266.

\bibitem{Tian}
        Y. Tian, 
        Congruent numbers and Heegner points. 
        Camb. J. Math. {\bf 2} (2014), no. 1, 117--161.
	
\bibitem{Tunnell}
    J. B. Tunnell, 
    A classical diophantine problem and modular forms. 
    Invent. Math. {\bf 72}, (1983), 323--334
    
\bibitem{Yoshida}
        S.-i. Yoshida, “Some variants of the congruent number problem, II”. Kyushu J. Math. {\bf 56}:1 (2002), 147--165.
	
\end{thebibliography}
}
\end{document}